\documentclass[11pt]{article}
\usepackage{amsfonts}
\usepackage{latexsym}
\usepackage{amsmath}
\usepackage{amssymb}
\usepackage{amsthm}
\usepackage{amsmath}
\usepackage{verbatim}
\usepackage{hyperref}
\usepackage{enumerate}
\usepackage{bibspacing}

\newtheorem{thm}{Theorem}[section]
\newtheorem*{thm*}{Theorem}

\newtheorem{corollary}[thm]{Corollary}

\newtheorem{lemma}[thm]{Lemma}
\newtheorem{prop}[thm]{Proposition}
\newtheorem*{prop*}{Proposition}
\newtheorem{proposition}[thm]{Proposition}

\newtheorem*{conj*}{Conjecture}
\newtheorem{defn}[thm]{Definition}
\newtheorem*{dfn*}{Definition}
\theoremstyle{definition}
\newtheorem{rem}[thm]{\textbf{Remark}}
\newtheorem*{rmk*}{Remark}
\newtheorem*{fact*}{Fact}

\theoremstyle{proof}

\DeclareMathOperator{\inj}{\textnormal{inj}}

\DeclareMathOperator{\interior}{\textnormal{int}}

\DeclareMathOperator{\Isom}{\textnormal{Isom}}
\DeclareMathOperator{\Stab}{\textnormal{Stab}}

\DeclareMathOperator{\supp}{\textnormal{spt}}

\newcommand{\Y}{\mathbf{Y}}
\newcommand{\T}{\mathbf{T}}

\newcommand{\norm}[1]{\left\Vert#1\right\Vert}
\newcommand{\snorm}[1]{\Vert#1\Vert}
\newcommand{\abs}[1]{\left\vert#1\right\vert}
\newcommand{\sabs}[1]{\vert#1\vert}

\newcommand{\R}{\mathbb{R}}
\newcommand{\Z}{\mathbb{Z}}

\renewcommand{\S}{\mathbb{S}}

\renewcommand{\H}{\mathcal{H}}
\newcommand{\I}{\mathcal{I}}

\newcommand{\M}{\mathbf{M}}

\newcommand{\eps}{\epsilon}

\newcommand{\Id}{\textrm{Id}}

\numberwithin{equation}{section}

\numberwithin{equation}{section}

\begin{document}

\renewcommand*{\thefootnote}{\fnsymbol{footnote}}

\author{Emanuel Milman\textsuperscript{$*$} and Joe Neeman\textsuperscript{$\dagger$}}
\footnotetext{$^*$Department of Mathematics, Technion-Israel Institute of Technology, Haifa 32000, Israel. Email: emilman@tx.technion.ac.il.}
\footnotetext{$^\dagger$Email: joe.neeman@gmail.com.}

\begingroup    \renewcommand{\thefootnote}{}    
    \footnotetext{The research leading to these results is part of a project that has received funding from the European Research Council (ERC) under the European Union's Horizon 2020 research and innovation programme (grant agreement No 101001677).}
\endgroup

\title{On the connectedness of a minimizing cluster's boundary}

\date{\nonumber} 
\maketitle

\begin{abstract}
We verify that an isoperimetric minimizing cluster on a simply connected homogeneous Riemannian manifold with at most one end always has connected boundary.  In particular, the boundary of a single-bubble isoperimetric minimizer on such manifolds must be connected, and hence all isoperimetric sets and their complements must be connected. This is demonstrably false without the simple connectedness assumption or the restriction on the number of ends. Handling general homogeneous Riemannian manifolds involves some unexpected complications. 
\end{abstract}

\section{Introduction}

Let $(M^n,g)$ denote an $n$-dimensional smooth Riemannian manifold, let $V$ denote the Riemannian volume measure, and let $\Isom(M)$ denote the group of isometries of $M$. Recall that $(M,g)$ is called a homogeneous Riemannian manifold if $\Isom(M)$ acts transitively on $M$, and that $(M,g)$ is necessarily complete in that case. 

Assume that $(M^n,g)$ is in addition complete and connected. 
The metric $g$ induces a geodesic distance $d$ on $M$, and the corresponding $k$-dimensional Hausdorff measure is denoted by $\H^k$. The perimeter (or surface area) of a Borel subset $U \subset M$ of locally finite perimeter is defined as $A (U) : = \H^{n-1} (\partial^* U)$, where $\partial^* U \subseteq \partial U$ is the reduced (measure-theoretic) boundary of $U$ \cite{MaggiBook}. Note that modifying $U$ on a null-set does not alter $\partial^* U$ and hence $A(U)$. 

A $q$-partition $\Omega = (\Omega_1, \ldots, \Omega_q)$ of $(M,g)$ is a $q$-tuple of Borel subsets $\Omega_i \subset M$ having locally finite perimeter, such that $\{\Omega_i\}$ are pairwise disjoint and $V(M \setminus \cup_{i=1}^q \Omega_i) = 0$. Note that the sets $\Omega_i$, called cells, are not required to be connected or non-empty. A $k$-tuple of pairwise disjoint cells $(\Omega_1,\ldots,\Omega_{k})$ so that $V(\Omega_i), A(\Omega_i) < \infty$ for all $i=1,\ldots,k$ is called a $k$-cluster, and its cells are called bubbles. Every $k$-cluster induces a partition by simply adding the ``exterior cell'' $\Omega_{k+1} := M \setminus \cup_{i=1}^{k} \Omega_i$; by abuse of notation, we will call the resulting $(k+1)$-partition  $\Omega = (\Omega_1,\ldots,\Omega_{k+1})$ a $k$-cluster (or $k$-bubble) as well. 
Of course, when $V(M) = \infty$ then necessarily $V(\Omega_{k+1}) = \infty$.

The volumes vector $V(\Omega)$ and total perimeter (or surface area) $A(\Omega)$ of a $q$-partition $\Omega$ are defined as
\[
	V (\Omega) := ( V(\Omega_1), \ldots, V(\Omega_q) )  ~,~ 
	A (\Omega) := \frac{1}{2} \sum_{i=1}^q A(\Omega_i) = \sum_{1\leq i <j \leq  q} \H^{n-1} (\Sigma_{ij}),
\]
where $\Sigma_{ij}:= \partial^* \Omega_i \cap \partial^* \Omega_j$ denotes the $(n-1)$-dimensional  interface between cells $\Omega_i$ and $\Omega_j$. 
The isoperimetric problem for $k$-clusters on $(M,g)$ consists of identifying those clusters $\Omega$ of prescribed volume $V(\Omega) = v$ which minimize the total perimeter $A(\Omega)$ -- such clusters are called isoperimetric minimizers. By modifying an isoperimetric minimizing cluster on a null set, we may and will assume that its cells $\Omega_i$ are open and satisfy $\Omega_i = \interior \overline{\Omega_i}$ and $\overline{\partial^* \Omega_i} = \partial \Omega_i$ for all $i$. Its boundary is then defined as $\Sigma = \cup_{i=1}^q \partial \Omega_i$. 
See Section \ref{sec:prelim} for additional information. 

\medskip

Our main result in this work is the following:

\begin{thm}\label{thm:intro-main}
Let $(M^n,g)$ be a smooth, connected, \textbf{simply connected}, homogeneous Riemannian manifold which is either \textbf{compact or has one end}. 
 Then an isoperimetric minimizing cluster on $(M^n,g)$  has  \textbf{connected} (and bounded) boundary $\Sigma$. 
\end{thm}

\begin{rem} \label{rem:intro-examples}
Theorem \ref{thm:intro-main} would be false without the simple connectedness assumption or restriction on the number of ends. To see this in dimension $n=1$, 
just consider the real line $\R$ (having 2 ends) or (non simply connected) sphere $\S^1$, whose isoperimetric minimizing $k$-clusters are obviously given by $k$ adjoined intervals (having disconnected boundary). To extend these counterexamples to higher dimension, consider the product $M^n = N^{n-1} \times \R$, where $N^{n-1}$ is a compact $(n-1)$-dimensional manifold; it was shown in \cite{Castro-Products,Gonzalo-Products,RitoreVernadakis-Products} that a single-bubble isoperimetric minimizer of large volume in $M^n$ is always of the form $N^{n-1} \times [a,b]$ and hence has disconnected boundary (this was first shown in \cite{Pedrosa-SphericalCylinders} for $N^{n-1} = \S^{n-1}$ the $(n-1)$-dimensional sphere). Choosing $N^{n-1}$ to in addition be a simply connected homogeneous Riemannian manifold (such as $\S^{n-1}$, $n \geq 3$) verifies that Theorem \ref{thm:intro-main} is false for $n$-dimensional simply connected manifolds with two ends (for $n \geq 3$; note that a two-dimensional simply connected non-compact manifold is necessarily homeomorphic to $\R^2$ and hence one-ended). The theorem is also false for any $n$-dimensional (non simply connected) torus $\R^n / \Z^n \simeq \S^1 \times \ldots \times \S^1$, since $\R^{n-1} / \Z^{n-1} \times [0,1/2]$ (whose boundary is disconnected) is known to be a single-bubble isoperimetric minimizer of volume $1/2$ (see e.g. \cite{EMilman-Slabs} and the references therein). 
\end{rem}

Since distinct connected components of a single-bubble isoperimetric minimizing cell cannot have intersecting boundaries (see Lemma \ref{lem:no-touch}), we deduce:
\begin{corollary} \label{cor:intro-main}
Let $(M^n,g)$ be as in Theorem \ref{thm:intro-main}, and let $\Omega_1$ be a single-bubble isoperimetric minimizer in $(M^n,g)$. Then $\Omega_1$, $M \setminus \overline{\Omega_1}$ and $\partial \Omega_1$ are all connected. 
\end{corollary}
\begin{rem}
It will be evident from the proof of Theorem \ref{thm:intro-main} that the connectedness of $\Omega_1$ in Corollary \ref{cor:intro-main} holds on \emph{any} connected homogeneous Riemannian manifold, regardless of its topology (but as the examples in Remark \ref{rem:intro-examples} demonstrate, this is not the case for $\partial \Omega_1$ nor $M \setminus \overline{\Omega_1}$, both of which may be disconnected). One would expect that the connectedness of $\Omega_1$ should follow from a standard argument, but this seems to require a delicate analysis involving a strong maximum principle which is valid not only on the regular part of $\partial \Omega_1$ -- see below. Showing that $\partial \Omega_1$ is connected requires some additional new ingredients involving multi-bubble clusters (even though Corollary \ref{cor:intro-main} only pertains to the single-bubble case). 
\end{rem}

Connectedness results for the \emph{boundary} of single-bubble isoperimetric minimizers are well-known in the literature for certain classes of Riemannian manifolds $(M^n,g)$. These results typically do not require any homogeneity nor global topological information, but rather local curvature information. For example, it is known that whenever the Ricci curvature of $(M^n,g)$ is strictly positive, then the boundary of single-bubble isoperimetric minimizers (and even just stable ones, having non-negative second variation of area modulo the volume constraint) must be connected, and this extends to manifolds with convex boundary and manifolds with non-negative Ricci curvature, unless the minimizer's boundary is totally geodesic \cite{SternbergZumbrun,BayleRosales}. Another example is that of Riemannian surfaces of revolution with control over their Gauss curvature (see \cite{BenjaminiCao} and \cite[Chapter 2]{Ritore-IsoperimetricBook}). A more general condition which ensures the connectedness of minimizers $\Omega_1$ (but not their boundary) is strict concavity of the (single-bubble) isoperimetric profile \cite[Theorem 3.17]{Ritore-IsoperimetricBook}, but in practice this is guaranteed by having strictly positive (generalized) Ricci curvature. To the best of our knowledge, Corollary \ref{cor:intro-main} is the first example of a connectedness result for an isoperimetric minimizer \emph{and its boundary} which does not require any curvature assumptions; however, it does not apply to merely stable boundaries.

\subsection{Ingredients of proof}

The assumption of simple connectedness provides us with the following crucial information, which may be of independent interest: 

\begin{thm}[A $2$-connected partition of a simply connected space has connected boundary]\label{thm:topology}
Let $X$ be a (connected) \textbf{simply connected} locally path-connected metric space. 
Let $\Omega = (\Omega_1,\ldots,\Omega_q)$ be a collection of pairwise disjoint non-empty open cells $\{\Omega_i\}_{i \in [q]}$ so that $\bar \Omega_{[q]} = X$, where $\bar \Omega_K := \cup_{i \in K} \overline{\Omega_i}$ for $K \subset [q] = \{1,\ldots,q\}$. Assume that for all $i \in [q]$, $\Omega_i$ is connected and satisfies $\Omega_i = \interior \overline{\Omega_i}$.

Then, if $\Sigma = \cup_{i \in [q]} \partial \Omega_i$ is disconnected, there exists a disjoint partition of the index set $[q]$ into three non-empty sets $I$, $\{\ell \}$ and $J$ so that $X$ is the \textbf{disjoint union} of $\bar \Omega_I$, $\Omega_\ell$ and $\bar \Omega_J$, and 
$\partial \Omega_\ell$ is the disjoint union of the two non-empty sets $\partial \bar \Omega_I$ and $\partial  \bar \Omega_J$. 
\end{thm}

\begin{rem}
The assertion of Theorem \ref{thm:topology} is false without the assumption of simple connectedness, 
as witnessed e.g.~by a partition of the cylinder $(\R / \Z) \times \R^{n-1}$ into $q \geq 3$ vertical strips -- such a partition's boundary is disconnected and the claim clearly fails for it. 
\end{rem}

Theorem \ref{thm:topology} was implicitly taken for granted in our previous work \cite[Lemma 9.3]{EMilmanNeeman-TripleAndQuadruple}, which dealt with the case that $X \in \{\R^n,\S^n\}$. However, in reviewing the argument there, we realized that it is worth emphasizing that simple connectedness is the crucial property being employed, motivating us to write up the present note.

\bigskip

The proof of Theorem \ref{thm:intro-main} now proceeds as follows. Given an isoperimetric minimizing cluster $\Omega$ on a connected homogeneous Riemannian manifold $(M,g)$, it is known that its boundary $\Sigma$ must be bounded. We pass to a refinement of this cluster $\tilde \Omega$ given by the connected components of its cells, i.e.~the connected components of $M \setminus \Sigma$. Using some known isoperimetric properties of homogeneous Riemannian manifolds, the refined partition $\tilde \Omega$ must still have finitely many cells. Furthermore, the assumption that $M$ has at most one end ensures that there is still at most one cell whose volume is infinite. Consequently, $\tilde \Omega$ is a genuine cluster, which must still be minimizing as it has the same boundary $\Sigma$ as $\Omega$. The idea now is a standard argument in geometric measure theory: if $\Sigma$ were disconnected, we could move one piece of the boundary isometrically until the first time it hits the second piece. Using the crucial assumption of simple connectedness, Theorem \ref{thm:topology} ensures the existence of an appropriate piece of the boundary to move so that the volumes of all cells are preserved, even in the presence of a single infinite volume cell; this ensures that the modified cluster remains an isoperimetric minimizer throughout this process. The contradiction at the time of collision is achieved using a strong maximum principle due to Ilmanen \cite{Ilmanen-SMP} and refined by Wickramasekera \cite{Wickra-SMP}, which applies not only to the regular part of the boundary (as in most previous installments of this argument), but to a certain singular part as well. 

\smallskip

We emphasize that for general homogeneous Riemannian manifolds, we don't know how to ensure that the first collision will occur at regular boundary points. Indeed, even if we pick a pair of closest (and hence regular) points $x,y$ on each piece of the disconnected boundary, and select a continuous path of isometries $\{\varphi_t\}$ which map $x$ to $y$ at time $t=1$, we cannot ensure in general that the first collision between the boundary pieces will not occur at an earlier time $t_0 < 1$. This is possible on certain particular classes of homogeneous Riemannian manifolds (such as $\delta$-homogeneous or generalized normal homogeneous ones \cite{DeltaHomogeneousManifolds,HomogeneousGeodesicsBook} like $\R^n$ or $\S^n$), possessing a transitive subgroup of isometries of maximal displacement, but to the best of our understanding, not for more general classes, not even for symmetric spaces of negative curvature like hyperbolic space. 
Consequently, handling general homogeneous Riemannian manifolds requires a fair amount of additional work.

\begin{rem}
As will be evident from the proof of Theorem \ref{thm:intro-main}, when $M$ is non-compact, we can remove the assumption that $(M,g)$ is one-ended if we assume that the minimizing $k$-cluster $\Omega$ satisfies that its infinite volume cell $\Omega_{k+1}$ has a single infinite volume connected component. However, in practice, we do not know how to ensure this without assuming that $M$ is one-ended. 
\end{rem}

The rest of this work is organized as follows. In Section \ref{sec:topology} we provide a proof of Theorem \ref{thm:topology}. In Section \ref{sec:prelim} we recall some geometric preliminaries and verify that an isoperimetric minimizing cluster on a homogeneous Riemannian manifold has finitely many connected components. In Section \ref{sec:GMT} we recall some facts from geometric measure theory and provide a proof of Theorem \ref{thm:intro-main}. 

\bigskip

\noindent
\textbf{Acknowledgments.} We thank Otis Chodosh for directing us to Ilmanen's strong maximum principle and providing helpful references, as well as Frank Morgan for suggesting to utilize the cone densities. We also thank Hank Green who found an omitted assumption in Lemma \ref{lem:good-cells} in the first version of this work.

\section{Topology - Proof of Theorem \ref{thm:topology}} \label{sec:topology}

The topological information we require is summarized in the following proposition. 

\begin{proposition} \label{prop:simply-connected}
Let $X$ be a connected, simply connected, locally path-connected metric space (in particular, this holds if $X$ is a simply connected paracompact topological manifold). 
Let $\Omega \subset X$ be an open, connected subset. Then $\partial \Omega$ is (dis)connected if and only if $X \setminus \Omega$ is (dis)connected. 
\end{proposition}
\begin{proof}
This is known to experts, and in the case of $X = \R^n$ was explicitly proved in \cite{ConnectednessOfBoundaryAndComplement}. 
However, the proof in \cite{ConnectednessOfBoundaryAndComplement} only uses that $X$ is a simply connected metric space 
with the property that any connected component of an open set is open, and any open connected set is path-connected; the latter two properties are ensured by the assumption that $X$ is locally path-connected \cite{Munkres-TopologyBook2ndEd}. 
\end{proof}

This is consistent with the case $q=2$ of Theorem \ref{thm:topology}: since $\{1,2\}$ cannot be partitioned into three non-empty disjoint sets, if $\Omega_1$ and $\Omega_2 = X \setminus \bar \Omega_1$ are both open and connected then $\Sigma = \partial \Omega_1 = \partial \Omega_2$ cannot be disconnected. Let us now extend this to arbitrary $q$. 

\begin{proof}[Proof of Theorem \ref{thm:topology}]
Since $\bar \Omega_{[q]} = X$ with $\{\Omega_i\}$ open and pairwise disjoint, it follows that $X \setminus \bar \Omega_I = \interior \bar \Omega_{[q] \setminus I}$ for all $I \subseteq [q]$, 
and consequently, $\partial \bar \Omega_I = \bar \Omega_I \cap \bar \Omega_{[q] \setminus I} = \partial \bar \Omega_{[q] \setminus I}$. In addition, since $\interior \overline{\Omega_i} = \Omega_i$ for all $i$, we see that $\partial \Omega_i = \partial \overline{\Omega_i} = \partial \bar \Omega_{\{i\}} = \partial \bar \Omega_{[q] \setminus \{i\}}$. 

We first claim that there exists a cell $\Omega_\ell$ with disconnected boundary. Otherwise, $\partial \Omega_i$ would be connected for all $i$. Since the union of two overlapping connected sets is connected, and since $\Sigma = \cup_i \partial \Omega_i$ is assumed to be disconnected, this would mean that there exists a non-empty proper subset $I \subset [q]$ so that $\cup_{i \in I} \partial \Omega_i$ and $\cup_{i \notin I} \partial \Omega_i$ are disjoint. 
But this would imply that $\bar \Omega_I$ and $\bar \Omega_{[q] \setminus I}$ are disjoint, and since both of these are closed, non-empty and their union is $X$, this would yield a contradiction to the connectedness of $X$ (which is implied by the simple-, and hence path-connectedness).  

Now, since $\Omega_\ell$ is open, connected and its boundary is disconnected, we know by Proposition \ref{prop:simply-connected} that its complement is also disconnected. Denote the connected components of $X \setminus \Omega_\ell$ by $\{C_k\}_{k \in \Lambda}$; we know that $|\Lambda| \geq 2$. 
Being the connected components of a closed set, $C_k$ are all closed (as closure preserves connectedness). 
Since $C_k$ are connected, each (connected) cell closure $\overline{\Omega_j}$ for $j \neq \ell$
 is either contained in $C_k$ or disjoint from it. Moreover, since $\cup_{j \neq \ell} \overline{\Omega_j} = \bar \Omega_{[q] \setminus\{\ell\}} = X \setminus \Omega_\ell = \cup_{k \in \Lambda} C_k$, it follows that $C_k = \cup_{j \in I_k} \overline{\Omega_j}$, where $\{I_k\}_{k \in \Lambda}$ is a disjoint partition of $[q] \setminus \{\ell \}$. Since $C_k \neq \emptyset$, it follows that $|I_k| \geq 1$. Note that this implies that $|\Lambda| \leq q-1 < \infty$.

Now, choose any $k \in \Lambda$, set $I = I_k$ (non-empty since $|I_k| \geq 1$) and $J = [q] \setminus (I \cup \{\ell\})$. We know that $\bar \Omega_I =  \cup_{j \in I} \overline{\Omega_j} = C_k$ and $\bar \Omega_J = \cup_{j \notin (I_k \cup \{\ell\}) } \overline{\Omega_j} = \cup_{m \neq k} C_m$; since $|\Lambda| \geq 2$ we see that the latter union is over a non-empty index set, and so $\bar \Omega_J \neq \emptyset$ and hence  $J$ is also non-empty (and is equal to $\cup_{m \neq k} I_m$). 

We thus confirm that $X$ is the disjoint union of $\bar \Omega_I$, $\Omega_\ell$ and $\bar \Omega_J$. 
It follows that $\partial \Omega_\ell = \partial \bar \Omega_{I \cup J}$ is the disjoint union of the two non-empty sets $\partial \bar \Omega_I = \partial \Omega_\ell \cap \bar \Omega_I$ and $\partial \bar \Omega_J = \partial \Omega_\ell \cap \bar \Omega_J$. 
\end{proof}

\section{Geometric preliminaries} \label{sec:prelim}

\subsection{Homogeneous Riemannian manifolds}

Note that by the classical Myers-Steenrod theorem \cite{PetersenBook2ndEd}, for any smooth Riemannian manifold $(M,g)$, $\Isom(M)$ is a Lie group acting smoothly on $M$, 
and its topology coincides with the compact-open topology, namely the uniform convergence of its action on compact subsets of $(M,g)$ (with respect to the geodesic distance $d$ on $(M,g)$). In particular, if $A,B$ are two compact subsets of $M$, then the map $\Isom(M) \ni \varphi \mapsto d(\varphi(A),B)$ is continuous, where $d(A,B) := \inf_{x \in A, y \in B} d(x,y)$.  

When $M$ is a homogeneous Riemannian manifold, $\Isom(M)$ acts transitively on $M$, and hence $M$ is homeomorphic to $\Isom(M) / \Stab(p)$, where $\Stab(p)$ is the closed stabilizer subgroup of $\Isom(M)$ which fixes (any) given point $p \in M$. When $M$ is in addition connected, it follows that $\Isom_0(M)$ also acts transitively on $M$, where $G_0$ denotes the identity (closed) connected component of a closed subgroup $G \subseteq \Isom(M)$. 

\begin{defn}
Let $(M,g)$ denote a smooth complete Riemannian manifold.
\begin{enumerate}
\item $(M,g)$ is said to have one end (or be one-ended) if $M \setminus K$ has a single unbounded connected component for any compact subset $K \subset M$. 
\item $(M,g)$ is said to have bounded geometry if its Ricci curvature is uniformly bounded below and the volume of geodesic balls of radius $1$ is uniformly bounded below away from $0$. 
\item $(M,g)$ is said to be co-compact if there exists a compact subset $K \subset M$ so that:
\[ M = \bigcup_{T \in \Isom(M)} T(K) . 
\] \end{enumerate}
\end{defn}

Note that a homogeneous Riemannian manifold is trivially co-compact and therefore has bounded geometry, since $\Isom(M)$ acts transitively on $M$. 

\subsection{Isoperimetric minimizers}

Given a smooth complete Riemannian manifold $(M^n,g)$, the metric $g$ induces a Riemannian volume measure $V$ and geodesic distance $d$ on $M$. 
Let $B(x,r)$ denote the geodesic ball in $M$ of radius $r > 0$ centered at $x \in M$. Given a measurable subset $E \subset M$ and $x \in M$, we denote:
\begin{equation} \label{eq:theta}
\theta(E,x,r) := \frac{V(E \cap B(x,r))}{V(B(x,r))} .
\end{equation}
We will frequently use that $E$ of locally finite perimeter satisfies
\begin{equation} \label{eq:half}
\partial^* E \subset E^{(1/2)} := \{ x \in M : \lim_{r \rightarrow 0^+} \theta(E,x,r) = 1/2\},
\end{equation}
the set of points of density $1/2$ of $E$ \cite[Corollary 15.8]{MaggiBook}. 

Given a $q$-partition $\Omega$ and an open set $U \subset M$, denote
\begin{equation} \label{eq:A-equiv}
A(\Omega ; U) := \sum_{1\leq i < j \leq q} \H^{n-1}(\partial^* \Omega_i \cap \partial^* \Omega_j \cap U) = \frac{1}{2} \sum_{i=1}^{q} \H^{n-1}(\partial^* \Omega_i \cap U)
\end{equation}
 (see \cite[Proposition 29.4]{MaggiBook} for the equivalence between the two alternative definitions of $A(\Omega ; U)$ above). We abbreviate $A(\Omega) = A(\Omega ; M)$. 
Recall that a $k$-cluster $\Omega$ is called an isoperimetric minimizer of volume $V(\Omega) \in [0,\infty)^k \times [0,\infty]$ if $A(\Omega') \geq A(\Omega)$ for any other $k$-cluster $\Omega'$ with $V(\Omega') = V(\Omega)$. In that case, we will also say that the $k$-cluster or $(k+1)$-partition $\Omega$ is minimizing.
 
 Given $\Lambda \geq 0$ and $r_0 \in (0,\infty]$, we will say $\Omega$ is $(\Lambda,r_0)$-perimeter minimizing in $U$ \cite{CicaleseLeonardiMaggi-AlmostPerimeterMinimizers} if for all $x \in U$, for all $r \in (0,r_0)$ such that $B(x,r) \Subset U$, and for any other $q$-partition $\Omega'$ such that $\Omega'_i \Delta \Omega_i \Subset B(x,r)$ for all $i=1,\ldots,q$, we have
 \[
 A(\Omega ; B(x,r)) \leq A(\Omega' ; B(x,r)) + \Lambda \sum_{i=1}^{q} V(\Omega_i' \Delta \Omega_i) . 
 \]
 We will simply say that $\Omega$ is locally almost perimeter minimizing if for any open relatively compact $U \subset M$ there exist $\Lambda^U , r^U_0 > 0$ such that $\Omega$ is $(\Lambda^U,r^U_0)$-perimeter minimizing in $U$; we will say that $\Omega$ is perimeter minimizing if it is $(0,\infty)$-perimeter minimizing in $M$. An isoperimetric minimizing cluster is locally almost perimeter minimizing \cite[Corollary 29.17]{MaggiBook}.  

\begin{lemma} \label{lem:good-cells} 
Let $\Omega$ be a $q$-partition of a Riemannian manifold $(M^n,g)$. Then:
\begin{enumerate}
\item \label{it:triv} $\H^{n-1}(\cup_{i=1}^{q} \partial \Omega_i) \geq  \H^{n-1}(\cup_{i=1}^{q} \partial^*\Omega_i)  = A(\Omega)$. 
\end{enumerate}
In addition, if $\Omega$ is locally almost perimeter minimizing (and in particular for isoperimetric minimizing clusters), then each cell $\Omega_i$ may be modified on a null-set (thereby not altering $V(\Omega)$ nor $A(\Omega)$) so that
\begin{enumerate}
\setcounter{enumi}{1}
\item \label{it:good1} $\Omega_i$ is open, $\overline{\partial^* \Omega_i} = \partial \Omega_i$ and $\Omega = (\Omega_1,\ldots,\Omega_{q})$ remains a $q$-partition. 
\item \label{it:good2} For any union $\Lambda_i$ of connected components of $\Omega_i$, $\Lambda_i = \interior \overline{\Lambda_i}$; in particular, this holds for $\Lambda_i = \Omega_i$ and for each individual connected component of $\Omega_i$. 
\item \label{it:good3} $\H^{n-1}(\partial \Omega_i \setminus \partial^* \Omega_i) = 0$. 
\item \label{it:good4} In particular, $\H^{n-1}(\cup_{i=1}^{q} \partial \Omega_i) = A(\Omega)$. 
\end{enumerate}
\end{lemma}
\begin{proof}
The first inequality in (\ref{it:triv}) is immediate since $\partial^* \Omega_i \subseteq \partial \Omega_i$. To see the second equality, we apply (\ref{eq:half}). 
Since $\Omega_a^{(1/2)} \cap \Omega_b^{(1/2)} \cap \Omega_c^{(1/2)} = \emptyset$ for distinct indices $a,b,c \in \{1,\ldots,q\}$ (otherwise the sum of densities at an intersection point would be at least $3/2$, a contradiction), it follows that $\partial^* \Omega_a \cap \partial^* \Omega_b \cap \partial^* \Omega_c = \emptyset$ as well. The inclusion-exclusion principle and (\ref{eq:A-equiv}) then imply that
\[
\H^{n-1}(\cup_{i=1}^{q} \partial^* \Omega_i) = \sum_{i=1}^{q} \H^{n-1}(\partial^* \Omega_i) - \sum_{1 \leq i < j \leq q} \H^{n-1}(\partial^* \Omega_i \cap \partial^* \Omega_j) = 2A(\Omega) - A(\Omega) = A(\Omega)
\]
(which holds also when $A(\Omega) = +\infty$ by exhausting $M$ with relatively compact sets). This concludes the proof of (\ref{it:triv}).

Each cell $\Omega_i$ is a set of locally finite perimeter, and so by \cite[Proposition 12.19]{MaggiBook} can be modified on a null set so that
\begin{equation} \label{eq:density0}
\overline{\partial^* \Omega_i} = \partial \Omega_i = \{ x \in M : \theta(\Omega_i,x,r) \in (0,1) \;\; \forall r > 0 \} .
\end{equation} 
By a standard Infiltration Lemma \cite[Lemma 4.5]{CicaleseLeonardiMaggi-AlmostPerimeterMinimizers} which applies to any locally almost perimeter minimizing partition, this means that there exists an $\eps > 0$ such that $\theta(\Omega_i , x , r) \in [\eps,1-\eps]$ for all $x \in \partial \Omega_i$ and sufficiently small $r > 0$ (depending on $x$). 
On the other hand, by Lebesgue's differentiation theorem, we have that $\lim_{r \rightarrow 0^+} \theta(\Omega_i , x , r) = \textbf{1}_{\Omega_i}(x) \in \{0,1\}$ for $\H^n$-almost every $x \in M$. It follows that $V(\partial \Omega_i) = 0$, and hence $\partial^* \Omega_i = \partial^* \interior \Omega_i$. 
Therefore, $\partial \interior \Omega_i \supseteq  \overline{\partial^* \interior \Omega_i} =  \overline{\partial^* \Omega_i} = \partial \Omega_i  \supseteq \partial \interior \Omega_i$ and so $\partial \interior \Omega_i = \partial \Omega_i$. 
Consequently, further modifying $\Omega_i$ by replacing it with its interior, we do not alter its topological boundary, and as this is again a null-set modification, it does not affect the reduced boundary $\partial^* \Omega_i$, so (\ref{eq:density0}) still holds. In fact, since $V(\partial \Omega_i) =0$ we also have:
\begin{equation} \label{eq:density}
 \partial \Omega_i =  \{ x \in M : \theta(\overline{\Omega_i},x,r) \in (0,1) \;\; \forall r > 0 \} . 
\end{equation}
The modified (open) cells $\Omega_i$ remain pairwise disjoint since otherwise, a non-empty intersection of a pair of cells would be an open set of positive volume, which is impossible since the original cells were pairwise disjoint and we only performed null-set modifications. The other requirements from a cluster do not change under null-set modifications. 

Now let $\Lambda_i$ be any union of connected components of $\Omega_i$. Since $M$ is locally connected and $\Omega_i$ is open, all of its connected components are also open, and so is $\Lambda_i$. Furthermore, since $\partial \Lambda_i \subset \partial \Omega_i$, 
(\ref{eq:density}) implies that for any $x \in \partial \Lambda_i$ we have $\theta(\overline{\Lambda_i},x,r) < 1$ for all $r > 0$. Consequently $\partial \Lambda_i \cap \interior \overline{\Lambda_i} = \emptyset$, and so if $x \in \interior \overline{\Lambda_i}$ then
  necessarily $x \in \Lambda_i$. In the other direction, since $\Lambda_i$ is already open we trivially have $\Lambda_i \subset \interior \overline{\Lambda_i}$. Therefore $\interior \overline{\Lambda_i} = \Lambda_i$, concluding the proof of statement (\ref{it:good2}).

Statement (\ref{it:good3}) is established in \cite[Theorem 3.6]{BellettiniKholmatov2018} 
for clusters in Euclidean space, but the proof (based on local density estimates) carries over to the Riemannian setting. 
  Consequently, $\H^{n-1}(\cup_{i=1}^{q} \partial \Omega_i) = \H^{n-1}(\cup_{i=1}^{q} \partial^* \Omega_i)$. Statement (\ref{it:good4}) immediately follows from (\ref{it:triv}) and (\ref{it:good3}). 
\end{proof}

We will henceforth always modify the cells of a locally almost perimeter minimizing $q$-partition $\Omega$ as in Lemma \ref{lem:good-cells}. Note that in that case, its boundary $\Sigma$ is uniquely defined as:
\[
\Sigma := \cup_{i=1}^{q} \overline{\partial^* \Omega_i} = \cup_{i=1}^{q} \partial \Omega_i ,
\]
and $M \setminus \Sigma$ is precisely the disjoint union of the open cells $\cup_{i=1}^{q} \Omega_i$. 

\begin{defn} \label{def:bounded}
A $q$-partition $\Omega$ is called bounded if $\Sigma$ is a bounded set.
\end{defn}

\begin{proposition} \label{prop:exist}
Let $(M,g)$ be a smooth complete connected co-compact Riemannian manifold. Then for any $k \geq 1$ and $v \in [0,\infty)^k \times [0,\infty]$ with $\sum_{i=1}^{k+1} v_i = V(M)$, there exists an isoperimetric minimizing $k$-cluster $\Omega$ in $(M,g)$ with $V(\Omega) = v$. 
\end{proposition}
\begin{proof}
The existence of an isoperimetric minimizing $k$-cluster in a co-compact manifold is due to Morgan (see the comments following \cite[Theorem 13.4]{MorganBook5Ed}), extending an argument of Almgren \cite{AlmgrenMemoirs}. 
See \cite[Theorem 2.3]{EMilmanNeeman-TripleAndQuadruple}, and also \cite[Theorem 4.25]{Ritore-IsoperimetricBook} in the single-bubble case. 
\end{proof}

\subsection{Isoperimetric profile}

The (single-bubble) isoperimetric profile $\I = \I_{(M,g)}$ of a Riemannian manifold $(M,g)$ is defined as the function $\I : [0,V(M)] \rightarrow \R_+$ given by
\[
\I(v) := \inf \{ A(E) : V(E) = v , \text{$E$ is a measurable subset of $M$}  \} .
\]
A measurable subset $E \subset M$ for which $A(E) = \I(V(E))$ is called a (single-bubble) isoperimetric minimizer of volume $V(E)$. Note that $E$ is a single-bubble isoperimetric minimizer if and only if the corresponding 1-cluster is an isoperimetric minimizer. 

\begin{lemma} \label{lem:positive}
Let $(M,g)$ be a smooth complete connected co-compact Riemannian manifold. Then $\I(v) > 0$ for every $v \in (0,V(M))$. 
\end{lemma}
\begin{proof}
Given $v \in (0,V(M))$, Proposition \ref{prop:exist} ensures the existence of a single-bubble isoperimetric minimizer $E$ of volume $v$. Since $v \in (0,V(M))$ and $M$ is connected, we must have $\partial E \neq \emptyset$. By (\ref{eq:density}), it follows that there exists $x \in \partial E$ and $r > 0$ smaller than the injectivity radius at $x$ so that $\theta(E,x,r) \in (0,1)$. Using the relative isoperimetric inequality inside $B(x,r)$ for the set $E \cap B(x,r)$ as in the proof of \cite[Lemma 3.4]{Ritore-IsoperimetricBook}, it follows that $E$ must have strictly positive perimeter, and hence $\I(v) > 0$. \end{proof}

\begin{proposition} \label{prop:small}
Let $(M^n,g)$ be a smooth complete connected Riemannian manifold with bounded geometry. Then its isoperimetric profile $\I$ is continuous on $(0,V(M))$, and there exist $v_0, c > 0$ so that $\I(v) \geq c v^{\frac{n-1}{n}}$ for all $v \in (0,v_0]$. 
\end{proposition}
\begin{proof}
The (local H\"older) continuity of $\I$ was established by Flores and Nardulli in \cite[Theorem 2]{FloresNardulli-ContinuityOfProfile}.
 The isoperimetric lower bound for small volumes can be found in \cite[Lemma 3]{Resende-ClustersInBoundedGeometry} based on an argument of Hebey \cite[Lemma 3.2]{Hebey-NonlinearAnalysisBook} (see also \cite[Lemma 4.26]{Ritore-IsoperimetricBook} in the co-compact case). Both assertions utilize the bounded geometry assumption. 
\end{proof}

\begin{proposition} \label{prop:bounded}
Let $(M,g)$ be a smooth complete connected Riemannian manifold with bounded geometry. Then any isoperimetric minimizing $k$-cluster $\Omega$ in $(M^n,g)$ is bounded. 
\end{proposition}
\begin{proof}
Boundedness of the finite volume cells of $\Omega$ follows by a classical argument of Almgren (e.g. \cite[Lemma 13.6]{MorganBook5Ed} or
\cite[Theorem 29.1]{MaggiBook}, cf. \cite[Lemma 4.27]{Ritore-IsoperimetricBook} in the case $k=1$), which is based on the validity of a single-bubble isoperimetric inequality for small volumes of the form $\I(v) \geq c v^{\frac{n-1}{n}}$ for all $v \in (0,v_0]$; the latter holds by Proposition \ref{prop:small}. 
Since $\Omega$ has at most one infinite volume cell, it follows that its boundary $\Sigma$ is necessarily bounded. 
\end{proof}

\begin{proposition} \label{prop:infinity}
Let $(M^n,g)$ be a connected non-compact homogeneous Riemannian manifold. Then $V(M) = \infty$ and $\liminf_{v \rightarrow \infty} \I(v) > 0$. 
\end{proposition}
\begin{proof}
This follows from a much more precise result of Pittet \cite[Theorem 2.1]{Pittet-ProfileOfHomogeneousManifolds}, who identified 3 possible asymptotic growth rates of $\I(v)$ as $v \rightarrow \infty$ (depending on algebraic properties of $\Isom_0(M)$); the minimal growth of $\I(v)$ corresponds to option (3) with $d=1$. 
\end{proof}

\subsection{Connected components of minimizers}

\begin{prop} \label{prop:finite-CCs}
Let $(M,g)$ be a smooth complete Riemannian manifold. 
Assume that:
\begin{enumerate}
\item \label{it:ass1} $\lim_{v \rightarrow 0^+} \frac{\I(v)}{v} = +\infty$; and
\item \label{it:ass2} For all $\eps > 0$, $\inf_{v \in (\eps , V(M)-\eps)} \I(v) > 0$. 
\end{enumerate}
Then for any isoperimetric minimizing $k$-cluster $\Omega$ on $(M,g)$ with boundary $\Sigma$, $M \setminus \Sigma$ has finitely many bounded connected components.
\end{prop} 
For $k=2$, this was verified by Hutchings \cite[Corollary 4.3]{Hutchings-StructureOfDoubleBubbles} on $\R^n$ and extended to all model spaces in \cite[Proposition 4.11]{CottonFreeman-DoubleBubbleInSandH}. This was extended on $\R^n$ and $\S^n$ to general $k \geq 2$ in \cite[Lemma 9.4]{EMilmanNeeman-TripleAndQuadruple} by adapting an argument of Morgan \cite[Theorem 2.3, Step 2]{MorganSoapBubblesInR2}, but the proof only uses assumptions (\ref{it:ass1}) and (\ref{it:ass2}). For completeness, let us sketch the argument. 
\begin{proof}[Sketch of Proof]
Assume in the contrapositive that the open $M \setminus \Sigma$ had an infinite (obviously countable) number of bounded connected components $\{\Theta_i\}_{i=1,\ldots,\infty}$. Denote $\Lambda_j  :=\cup_{i=j}^\infty \Theta_i$. Since the cluster's total perimeter is finite, $\sum_{i=1}^\infty A(\Theta_i) \leq 2 \H^{n-1}(\Sigma) < \infty$ and hence $A(\Lambda_j) \leq \sum_{i=j}^\infty A(\Theta_i) \rightarrow 0$ as $j \rightarrow \infty$. In particular, $\I(V(\Theta_j)) \leq A(\Theta_j) \rightarrow 0$ as $j \rightarrow \infty$, and hence assumption (\ref{it:ass2}) when $V(M) = \infty$ implies that $V(\Theta_j) \rightarrow 0$; alternatively, disjointness of $\Theta_i$ implies the same conclusion when $V(M) < \infty$. 
Assumption (\ref{it:ass1}) therefore implies that $V(\Theta_i) \leq A(\Theta_i)$ for all sufficiently large $i$, and so for sufficiently large $j$, $V(\Lambda_j) = \sum_{i=j}^\infty V(\Theta_i) \leq \sum_{i=j}^\infty A(\Theta_i) \rightarrow 0$ as $j \rightarrow \infty$. Applying assumption (\ref{it:ass1}) once more, now to $\Lambda_j$, this means that $\lim_{j \rightarrow \infty} A(\Lambda_j) / V(\Lambda_j) \geq \lim_{j \rightarrow \infty} \I(V(\Lambda_j)) / V(\Lambda_j)= +\infty$.  
It follows that for any $C,J \geq 1$ and $\eps > 0$, there exists $j \geq J$ such that $V(\Lambda_j) \leq \eps$ and $A(\Lambda_j) \geq C V(\Lambda_j)$. From here on the proof proceeds exactly as in \cite[Lemma 9.4]{EMilmanNeeman-TripleAndQuadruple}. 
\end{proof}

\begin{rem} \label{rem:finite-CCs}
Note that Lemma \ref{lem:positive} and Propositions \ref{prop:small} and \ref{prop:infinity} imply that assumptions (\ref{it:ass1}) and (\ref{it:ass2}) are satisfied for any connected homogeneous Riemannian manifold. 
\end{rem}

\section{Geometric Measure Theory -- proof of Theorem \ref{thm:intro-main}} \label{sec:GMT}

We will use the following facts from Geometric Measure Theory. For background 
on varifolds, we refer to \cite{Simon-IntroToGMT}. All of our varifolds will be multiplicity one rectifiable $(n-1)$-varifolds $Z$, inducing a locally finite mass measure $\norm{Z} = \H^{n-1} \llcorner \Sigma_Z$, where $\Sigma_Z$ is a countably $(n-1)$-rectifiable $\H^{n-1}$-measurable subset of the ambient smooth Riemannian manifold $(M^n,g)$; we will write $Z = \abs{\Sigma_Z}$. 
For example, $Z = \abs{\partial^* \Omega}$ is such a varifold for any Borel set $\Omega \subset M$ with locally finite perimeter \cite[Chapter 3, Theorem 4.4]{Simon-IntroToGMT}, and more generally, $Z = \abs{\cup_{i=1}^{q} \partial^* \Omega_i}$ for any $q$-partition $\Omega$ in $M$. 
When $M^n$ and $N^n$ are Euclidean spaces and $T : M^n \rightarrow N^n$ is an invertible linear map, we denote $T_* Z = |T \Sigma_Z|$. 
We will say that $Z$ has locally bounded generalized mean-curvature and no boundary in an open $U \subset M$ if $\abs{\int \text{div}_{\Sigma_Z} X d\norm{Z} }\leq C_W \int_W|X| d\norm{Z}$ 
for every $C^1$ vector field $X$ compactly supported in $W \Subset U$; when $C_W = 0$ for all $W \Subset U$, $Z$ is called stationary (in $U$). 
Consequently  \cite[Chapter 4, Remark 4.9]{Simon-IntroToGMT}, when this condition holds with $U = M$, we may assume that $\Sigma_Z = \supp \norm{Z}$, the support of the mass measure $\norm{Z}$. 
For example, if $\Omega$ is a minimizing $k$-cluster in $M^n$, whose cells have been modified on null-sets as in Lemma \ref{lem:good-cells}, then the varifold $Z = \sabs{\cup_{i=1}^{k+1} \partial^* \Omega_i} = \sabs{\cup_{i=1}^{k+1} \partial \Omega_i}$ satisfies all of the above properties \cite{EMilmanNeeman-GaussianMultiBubble,EMilmanNeeman-TripleAndQuadruple};  in addition, $\supp \norm{Z} = \cup_{i=1}^{k+1} \partial \Omega_i$ is $(\M,\eps,\delta)$ minimizing in the sense of Almgren \cite{AlmgrenMemoirs,MorganBook5Ed,CES-RegularityOfMinimalSurfacesNearCones}.
  
We summarize what we require below in the following theorem. We denote by $\inj_x M$ the injectivity radius of $(M,g)$ at $x \in M$, by $\exp_x : T_x M \rightarrow M$ the exponential map, and by $B^{n-1}(r)$ the centered Euclidean ball of radius $r > 0$ in $\R^{n-1}$. Convergence of measures in $w^*$ means in duality with compactly supported continuous functions. 

\begin{thm} \label{thm:GMT}
Let $\Omega$ be a locally almost perimeter minimizing $q$-partition in $(M^n,g)$, let $\Sigma = \cup_{i=1}^q \overline{\partial^* \Omega_i}$ and $x \in \Sigma$. Then for any sequence $\lambda_j \searrow 0$, there exists a subsequence $\{j_k\}$ and a $q$-partition $\Omega^x$ of $T_x M$, called a ``blow-up" of $\Omega$ at $x$, such that:
\begin{enumerate}
\item \label{it:set-blowup}
 Denoting $\Omega_i^{x,\lambda} := \{ v \in T_x M : \lambda |v| < \inj_x M ~,~ \exp_x(\lambda v) \in \Omega_i \}$, we have that $\H^n \llcorner \Omega^{x,\lambda_{j_k}}_i \rightarrow \H^n \llcorner \Omega^{x}_i$ in $w^*$ as $k \rightarrow \infty$, for all $i=1,\ldots,q$. 
\item \label{it:varifold-blowup}
As multiplicity one rectifiable $(n-1)$-varifolds in $T_x M$, we have the following convergence as $k \rightarrow \infty$:
\[
 Z^{x,\lambda_{j_k}} := \sabs{\cup_{i=1}^q \partial^* \Omega_i^{x,\lambda_{j_k}}} \rightarrow \sabs{\cup_{i=1}^q \partial^* \Omega_i^{x}}  =: Z^x .
\]
In particular, $\snorm{Z^{x,\lambda_{j_k}}} \rightarrow \norm{Z^x}$ in $w^*$. 
\item The (possibly empty) cells $\Omega^x_i$ are open, $\partial \Omega^x_i = \overline{\partial^* \Omega^x_i}$, $\H^{n-1}(\partial \Omega^x_i  \setminus \partial^* \Omega^x_i) = 0$ and $\supp \norm{Z^x} = \cup_{i=1}^q \partial \Omega^x_i \neq \emptyset$. 
\item \label{it:cones}
$\Omega^x_i$ and $Z^x$ are cones: for all $\lambda > 0$, $\lambda \Omega^x_i = \Omega^x_i$ for all $i=1,\ldots,q$ and $\lambda_* Z^x = Z^x$. 
\item \label{it:mul1}
$Z^x$ is stationary, $\Omega^x$ is perimeter minimizing, and $\Sigma^x := \supp \norm{Z^x} \subset T_x M$ is $(\M,0,\infty)$ minimizing in the sense of Almgren.  
\item \label{it:density}
 The density $\Theta(\Sigma,x) := \lim_{r \rightarrow 0^+} \frac{\H^{n-1}(\Sigma \cap B(x,r))}{\H^{n-1}(B^{n-1}(r))}$ exists, is finite, and coincides with $\Theta(\Sigma^x,0)$. 
 \item \label{it:CES}
$\Sigma^x$ is the disjoint union of sets $\Sigma^{x,1}$, $\Sigma^{x,2}$, $\Sigma^{x,\geq 3}$, satisfying:
\begin{enumerate}
\item \label{it:a} $\Sigma^{x,1}$ is a locally-finite union of embedded $(n-1)$-dimensional $C^{\infty}$ manifolds, and for every $v \in \Sigma^{x,1}$, $\Sigma^x$ around $v$ is locally $C^\infty$ diffeomorphic to $\{0\} \times \R^{n-1}$.
\item  \label{it:b} $\Sigma^{x,2}$ is a locally-finite union of embedded $(n-2)$-dimensional $C^{\infty}$ manifolds, and for every $v \in \Sigma^{x,2}$, $\Sigma^x$ around $v$ is locally $C^\infty$ diffeomorphic to $\Y \times \R^{n-2}$.
\item $\Sigma^{x,\geq 3}$ is closed and has locally finite $\H^{n-3}$ measure. 
\end{enumerate}
Here $\Y$ denotes the planar cone consisting of $3$ half-lines meeting at the origin in $120^\circ$ angles. If $n = 1$ then $\Sigma^{x,2} = \emptyset$ and if $n\leq 2$ then $\Sigma^{x, \geq 3} = \emptyset$. 
\end{enumerate}
\end{thm}
\begin{proof}
The statement is classical for $q=2$ (in which case $\Sigma^{x,2} = \emptyset$) -- see for instance \cite[Theorem 28.6]{MaggiBook}. 
For general $q \geq 2$, we argue as follows. 
The subsequential existence of a blow-up of each $\Omega_i$ is a consequence of a well-known compactness argument, 
based on the (weighted) monotonicity formula for locally almost perimeter minimizing partitions  \cite[Chapter 4, Theorem 3.20 and Remark 3.21]{Simon-IntroToGMT}, because their associated boundary varifolds  have locally bounded generalized mean curvature and no boundary (see \cite[Corollary 4.8]{CicaleseLeonardiMaggi-AlmostPerimeterMinimizers}, whose local first-variation proof applies equally in the Riemannian setting).
Consequently, the perimeters of $\Omega^{x,\lambda}_{i}$ are locally uniformly bounded (for all $i=1,\ldots,q$ and sufficiently small $\lambda > 0$), and we may sequentially pass to appropriate subsequences which ensure the $w^*$ convergence $\H^n \llcorner \Omega^{x,\lambda_{j_k}}_i \rightarrow \H^n \llcorner \Omega^{x}_i$  for some set $\Omega^x_i$ of locally finite perimeter in $T_x M$ \cite[Corollary 12.27]{MaggiBook}. 
In particular, this implies $L^1_{loc}$ subsequential convergence of $\Omega^{x,\lambda}_i$ to $\Omega^x_i$, and since $\sum_{i=1}^q \textbf{1}_{\Omega^{x,\lambda}_i} = \textbf{1}$ almost-everywhere on every compact subset of $T_x M$ for all sufficiently small $\lambda > 0$, this persists in the limit and we obtain a limiting partition $\Omega^x$. Moreover, by \cite[Theorem 4.13]{CicaleseLeonardiMaggi-AlmostPerimeterMinimizers}, the limiting partition is in fact perimeter minimizing and moreover a cone; this Euclidean result extends to the Riemannian setting by working in normal coordinates, since the rescaled metrics converge smoothly to the Euclidean metric.
Consequently, Lemma \ref{lem:good-cells} (2) applies, and so modifying $\{\Omega^x_i\}$ by null-sets, we may ensure that $\{\Omega^x_i\}$ are the open cells of a partition, satisfying $\overline{\partial^* \Omega^x_i} = \partial \Omega^x_i$ 
for all $i=1,\ldots,q$. 
At least 2 cells among $\{\Omega^x_i\}$ are non-empty, since otherwise we would not have $x \in \Sigma$ by the Infiltration Lemma \cite[Lemma 30.2]{MaggiBook}. 
Therefore $\cup_{i=1}^q \partial \Omega^x_i$ is non-empty.

Next, the same compactness argument also gives convergence of the
local total perimeter measures; see the proof of
\cite[Theorem 21.14]{MaggiBook} and the discussion preceding
\cite[Theorem 4.9]{CicaleseLeonardiMaggi-AlmostPerimeterMinimizers}.
By lower semicontinuity of each cell's perimeter and convergence
of their finite sum, we obtain $\H^{n-1}\llcorner\partial^*\Omega_i^{x,\lambda_{j_k}}
 \rightarrow
 \H^{n-1}\llcorner\partial^*\Omega_i^x$ in $w^*$.  Together with $L^1_{\mathrm{loc}}$ convergence, Reshetnyak's
continuity theorem \cite[Section 20.3]{MaggiBook}, applied to
the Gauss--Green measures, therefore yields the varifold convergence
\[
 Z^{x,\lambda_{j_k}}
 =\frac12\sum_{i=1}^q
   \sabs{\partial^*\Omega_i^{x,\lambda_{j_k}}}
 \longrightarrow
 \frac12\sum_{i=1}^q\sabs{\partial^*\Omega_i^x}
 =\sabs{\bigcup_{i=1}^q\partial^*\Omega_i^x}
 =:Z^x .
\]
The first equality is understood on every fixed relatively compact
subset of $T_x M$, for all sufficiently large $k$.
These equalities follow because almost every interface point
belongs to exactly two reduced boundaries
\cite[Proposition 29.4]{MaggiBook}. In particular, $Z^x$ has
multiplicity one.

Lemma \ref{lem:good-cells} gives
\[
 \supp\norm{Z^x}
 =\bigcup_{i=1}^q\partial\Omega_i^x=:\Sigma^x,
 \qquad
 \H^{n-1}(\partial\Omega_i^x\setminus\partial^*\Omega_i^x)=0.
\]
Since $\Omega^x$ is perimeter minimizing and its cells are cones,
$Z^x$ is stationary and conical. Moreover, the proof of
\cite[Proposition 13.8]{MorganBook5Ed}, without the
volume-correction step, shows that $\Sigma^x$ is
$(\M,0,\infty)$ minimizing.
This verifies parts (\ref{it:set-blowup}) through (\ref{it:mul1}). 

Part (\ref{it:density}) follows from \cite[Chapter 4, Remark 3.21, and Chapter 8 (5.2)]{Simon-IntroToGMT}. 
Part (\ref{it:CES}) follows by \cite[Theorem 3.15 and its proof, and Remark 1.5]{CES-RegularityOfMinimalSurfacesNearCones}; 
as explained in the proof, part (\ref{it:a}) is due to Allard \cite{Allard-FirstVariationOfVarifold} and Almgren \cite{AlmgrenMemoirs}, and part (\ref{it:b}) is due to Simon \cite{Simon-Sigma2}. 
In our setting, $Z^x$ has an associated cycle structure (being the boundary of a partition), which allows a further decomposition of $\Sigma^{x,\geq 3}$ -- see Remark \ref{rem:CES} below. 
\end{proof}
 
We will crucially also need the following strong maximum principle due to Wickramasekera \cite[Theorem 19.1]{Wickra-SMP}
\begin{thm}[Wickramasekera] \label{thm:SMP}
If $Z_1,Z_2$ are two stationary integral $(n-1)$-varifolds on a smooth complete Riemannian manifold $(M^n,g)$ such that $\H^{n-2}(\supp \norm{Z_1} \cap \supp \norm{Z_2}) = 0$ then $\supp \norm{Z_1} \cap \supp \norm{Z_2} = \emptyset$. 
\end{thm}

\begin{rem} \label{rem:CES}
The above strong maximum principle supersedes a prior one by Ilmanen \cite{Ilmanen-SMP}, which required knowing that $\H^{n-3}(\supp \norm{Z_1} \cap \supp \norm{Z_2})  = 0$. 
We could also employ the latter weaker version, but we would then need to use the additional structure of $\Sigma^{x,\geq 3}$ due to Colombo--Edelen--Spolaor \cite[Theorem 3.15, Remark 1.5]{CES-RegularityOfMinimalSurfacesNearCones}. These authors showed, under an associated cycle structure assumption on $\Sigma^x$ (which holds in our setting since $\Sigma^x = \cup_{i=1}^q \partial \Omega^x_i$), that $\Sigma^{x,\geq 3}$ further decomposes into a disjoint union of $\Sigma^{x,3}$ and $\Sigma^{x,\geq 4}$. Here $\Sigma^{x,3}$ is  a locally-finite union of embedded $(n-3)$-dimensional $C^{1,\alpha}$ manifolds such that for every $v \in \Sigma^{x,3}$, $\Sigma^x$ around $v$ is locally $C^{1,\alpha}$ diffeomorphic to $\T \times \R^{n-3}$, where $\T$ denotes the cone over the edges of a regular tetrahedron in $\R^3$; and $\Sigma^{x,\geq 4}$ is closed with $\H^{n-3}(\Sigma^{x,\geq 4}) = 0$ (and in fact is $(n-4)$-rectifiable and has locally finite $\H^{n-4}$ measure by results of Naber--Valtorta \cite{NaberValtorta-MinimizingHarmonicMaps, NaberValtorta-SingularSet}). 
\end{rem}

We are now ready to give a proof of Theorem \ref{thm:intro-main}. 

\begin{proof}[Proof of Theorem \ref{thm:intro-main}]
Recall that $(M^n,g)$ is a connected homogeneous Riemannian manifold, and let $\Omega^0$ be an isoperimetric minimizing $k$-cluster in $(M^n,g)$. As per our convention in Lemma \ref{lem:good-cells}, we modify $\Omega^0$ by null-sets so that all of its cells $\Omega^0_j$ are open and satisfy $\partial \Omega^0_j = \overline{\partial^* \Omega^0_j}$. Lemma \ref{lem:good-cells} ensures that $A(\Omega^0) = \H^{n-1}(\Sigma)$ where $\Sigma = \cup_{j=1}^{k+1} \partial \Omega^0_j$, that $M \setminus \Sigma$ coincides with the disjoint union $\cup_{j=1}^{k+1} \Omega^0_j$, and that each (open) connected component $\Omega_i$ of $\Omega^0_j$ satisfies $\Omega_i = \interior \overline{\Omega_i}$.  By removing empty cells if necessary, we may assume that all cells are non-empty. 
 
 By Proposition \ref{prop:bounded}, $\Sigma$ is a bounded set and hence compact. 
Consider the bounded connected components of $M \setminus \Sigma$ -- by Proposition \ref{prop:finite-CCs} and Remark \ref{rem:finite-CCs} there are only a finite number of those. If $M$ is compact then there are no unbounded components, whereas if $M$ is one-ended then there is exactly one single unbounded component. Denoting all components by $\{\Omega_i \}_{i=1}^{N+1}$ (with the unbounded component, if there is one, placed last), the partition $\Omega = (\Omega_1,\ldots, \Omega_{N+1})$ is clearly an $N$-cluster with the same topological boundary $\cup_{i=1}^{N+1} \partial \Omega_i = \Sigma$ as that of $\Omega^0$. By Lemma \ref{lem:good-cells} we have $A(\Omega^0) = \H^{n-1}(\Sigma) \geq  A(\Omega)$, but since $\Omega$ is a refinement of $\Omega^0$ we cannot have a strict inequality, as this would contradict the minimality of $\Omega^0$. Consequently we have $A(\Omega^0) = A(\Omega)$ and so $\Omega$ is an isoperimetric minimizing $N$-cluster itself. Recall that its cells $\Omega_i$ are open, connected, non-empty, pairwise disjoint, satisfy $\Omega_i = \interior \overline{\Omega_i}$ and $M = (\cup_{i=1}^{N+1} \Omega_i)\cup \Sigma = \cup_{i=1}^{N+1} \overline{\Omega_i}$. 
 
Our goal is to show that whenever $M$ is in addition simply connected then $\Sigma$ must be connected. Assume otherwise. By Theorem \ref{thm:topology}, there exists a disjoint partition of the index set $\{1,\ldots,N+1\}$ into three non-empty sets $I$, $\{\ell \}$ and $J$ so that $M$ is the disjoint union of $\bar \Omega_I$, $\Omega_\ell$ and $\bar \Omega_J$, where recall $\bar \Omega_K := \cup_{i \in K} \overline{\Omega_i}$, and $\partial \Omega_\ell$ is the disjoint union of the two non-empty sets $\partial \bar \Omega_I$ and $\partial \bar \Omega_J$. 
Since $\partial \Omega_\ell \subset \Sigma$ is compact, these two disjoint compact subsets are a positive distance apart. 
Without loss of generality we may assume that $N+1 \notin I$ (by replacing $I$ with $J$ if necessary); recall that we have $V(\Omega_i) < \infty$ for all $i=1,\ldots,N$, but the last cell $\Omega_{N+1}$ may have infinite volume.

Choose any two points $x \in \partial \bar \Omega_I$ and $y \in \partial \bar \Omega_J$. Since $M$ is connected, $\Isom_0(M)$ acts transitively on $M$, and so there exists an isometry $\varphi \in \Isom_0(M)$ so that $\varphi(x) = y$. As $\Isom_0(M)$ is path-connected, there exists a continuous path $[0,1] \ni t \mapsto \varphi_t \in \Isom_0(M)$ so that $\varphi_0 = \Id$ and $\varphi_1 = \varphi$.  It follows that the function $[0,1] \ni t \mapsto d(t) := d(\varphi_t(\partial \bar \Omega_I),\partial \bar \Omega_J)$ is continuous, with $d(0) > 0$ and $d(1) = 0$, so there is a first time $t_0 \in (0,1]$ such that $d(t_0) = 0$. 

 Now define, for all $t \in [0,t_0]$, the $N$-cluster $\Omega^t$ whose open cells are given by:
\[
\Omega^t_i = \begin{cases} \varphi_t(\Omega_i) & i \in I \\ \Omega_i & i \in J \\ M \setminus (\varphi_t(\bar \Omega_I) \cup \bar \Omega_J) & i = \ell \end{cases} . 
\]

We first claim that $V(\Omega^t) = V(\Omega)$ for all $t \in [0,t_0]$. We obviously have $V(\Omega^t_i) = V(\Omega_i)$ for all $i \neq \ell$,  but one has to be careful with $V(\Omega^t_\ell)$ when $V(M) = \infty$. Since $N+1 \notin I$ then $V(\bar \Omega^t_I ) = V(\bar \Omega_I) < \infty$, and either $\ell = N+1$, in which case $V(\Omega^t_\ell) = \infty$ for all $t$, or else $N+1 \in J$ and hence $M \setminus \bar \Omega^t_J = M \setminus \bar \Omega_J$ has finite volume, verifying that $V(\Omega^t_\ell) = V(M \setminus \bar \Omega^t_J) - V(\bar \Omega^t_I) =  V(M \setminus \bar \Omega_J) - V(\bar \Omega_I) < \infty$ remains constant. Note that we crucially used here the fact that there is at most a single cell with infinite volume, and this is where our proof would break down if $M$ had more than one end. 

Write $\Sigma = \Sigma_I \cup \Sigma_J$, where  $\Sigma_I := \Sigma \cap \bar \Omega_I$ and  $\Sigma_J := \Sigma \cap \bar \Omega_J$.
Denote by $\Sigma^t = \cup_{i=1}^{N+1} \partial \Omega^t_i$ the topological boundary of $\Omega^t$, $t \in [0,t_0]$. 
Clearly $\Sigma^t = \varphi_t(\Sigma_I) \cup \Sigma_J$, and as these two parts of the boundary are disjoint before their collision at time $t=t_0$, we have that $\H^{n-1}(\Sigma^t) = \H^{n-1}(\Sigma)$ for all $t \in [0,t_0)$. Since $A(\Omega^t) \leq \H^{n-1}(\Sigma^t)$ with equality at $t=0$ by Lemma \ref{lem:good-cells}, and since $V(\Omega^t) = V(\Omega)$ and $\Omega$ is minimizing, it follows that $\Omega^t$ must also be minimizing for all $t \in [0,t_0)$ and that $A(\Omega^t) = A(\Omega)$ in that range (this can also easily be deduced by directly inspecting the evolution of $\partial^* \Omega^t_i$, and in particular that of $\partial^* \Omega^t_\ell$).
At the time of collision $t=t_0$, we have
\begin{align*}
 \H^{n-1}(\Sigma^{t_0}) & = \H^{n-1}(\varphi_{t_0}(\Sigma_I)) + \H^{n-1}(\Sigma_J) - \H^{n-1}(\varphi_{t_0}(\Sigma_I) \cap \Sigma_J) \\
 & = \H^{n-1}(\Sigma) - \H^{n-1}(\varphi_{t_0}(\Sigma_I) \cap \Sigma_J) \\
 & = \H^{n-1}(\Sigma) - \H^{n-1}(\varphi_{t_0}(\partial \bar \Omega_I) \cap \partial \bar \Omega_J) .
 \end{align*}
 Consequently, we must have 
 \begin{equation} \label{eq:no-touch}
 \H^{n-1}(\varphi_{t_0}(\partial \bar \Omega_I) \cap \partial \bar \Omega_J)  = 0,
 \end{equation}
  otherwise we
  will obtain that $V(\Omega^{t_0}) = V(\Omega)$ with $A(\Omega^{t_0}) \leq \H^{n-1}(\Sigma^{t_0}) < \H^{n-1}(\Sigma)$, contradicting the minimality of $\Omega$. It follows that $\Omega^t$ at the collision time $t=t_0$ is still minimizing. 

Let $y \in \varphi_{t_0}(\partial \bar \Omega_I) \cap \partial \bar \Omega_J \subset \partial \Omega_\ell$ be an arbitrary collision point, and write $y = \varphi_{t_0}(x)$ with $x \in \partial \bar \Omega_I$. Applying Theorem \ref{thm:GMT}, let $\Omega^x$ be a blow-up of $\Omega$ at $x$ corresponding to some sequence of dilates $\lambda_j \searrow 0$, and let $Z^x = |\cup_{i \in I \cup \{\ell\}} \partial^*  \Omega^x_i|$ be the corresponding boundary varifold in $T_x M$, and $\Sigma^x := \supp \norm{Z^x} = \cup_{i \in I} \partial \Omega^x_i$. Applying Theorem \ref{thm:GMT} again, there exists a blow-up $\Omega^y$ of $\Omega$ at $y$ corresponding to dilates $\lambda_{j_k}$ for an appropriate  subsequence $\{j_k\}$, and we denote by $Z^y = |\cup_{j \in J \cup \{\ell\}} \partial^*  \Omega^y_j|$ the corresponding boundary varifold in $T_y M$, and  $\Sigma^y := \supp \norm{Z^y} = \cup_{j \in J} \partial \Omega^y_j$.
Both of these multiplicity one varifolds are stationary cones, and since $\varphi_{t_0}$ is an isometry, so is $Z^x_{t_0} := (d\varphi_{t_0})_*(Z^x)$ in $T_y M$ with $\supp \norm{Z^x_{t_0}} = d\varphi_{t_0}(\Sigma^x)$, and all the properties ensured by Theorem \ref{thm:GMT} continue to hold for $Z^x_{t_0}$. 

Being (non-empty) closed cones in $T_y M$, we have $0 \in \Sigma^y \cap d\varphi_{t_0}(\Sigma^x)$, and hence Theorem \ref{thm:SMP} implies that $\H^{n-2}(\Sigma^y \cap d\varphi_{t_0}(\Sigma^x)) > 0$. We denote by $\Sigma^{x,\leq2 } = \Sigma^{x,1} \cup \Sigma^{x,2}, \Sigma^{x, \geq 3}$ and $\Sigma^{y, \leq 2} = \Sigma^{y,1} \cup \Sigma^{y,2}, \Sigma^{y,\geq 3}$ the corresponding decompositions of $\Sigma^x$ and $\Sigma^y$ ensured by Theorem \ref{thm:GMT}. 
 Since $\H^{n-2}(\Sigma^{y,\geq 3}) = \H^{n-2}(\Sigma^{x,\geq 3}) = 0$, there exists $v \in \Sigma^{y,\leq 2} \cap d\varphi_{t_0}(\Sigma^{x,\leq 2})$. 

We now claim that in fact $v \in \Sigma^{y,1} \cap d\varphi_{t_0}(\Sigma^{x,1})$. Assume in the contrapositive and without loss of generality that $v \in \Sigma^{y,2}$ (an identical argument applies to the case that $v \in d\varphi_{t_0}(\Sigma^{x,2})$), and recall the definition (\ref{eq:theta}). Note that $\lim_{r \rightarrow 0^+} \theta(\bar \Omega^{y}_J, v, r) \geq 2/3$, since there are only $3$ non-empty cells of $\Omega^y$ among the possible ones indexed by $J \cup \{ \ell \}$ in a small-enough neighborhood of $v \in \Sigma^{y,2}$, and each of these cells has density $1/3$ by stationarity of $Z^y$. Similarly,  $\lim_{r \rightarrow 0^+} \theta(d\varphi_{t_0}(\bar \Omega^{x}_I), v, r) \geq 1/2$, since there are only $2$ or $3$ non-empty cells of $\Omega^x$ among the possible ones indexed by $I \cup \{ \ell \}$ in a small-enough neighborhood of $w = d\varphi_{t_0}^{-1}(v) \in \Sigma^{x,\leq 2}$ (depending on whether $w \in \Sigma^{x,1}$ or $w \in \Sigma^{x,2}$), and each of these cells has density $1/2$ or $1/3$, respectively. Consequently, 
\begin{align*}
1 & \geq \lim_{r \rightarrow 0^+} \theta(\bar \Omega^{y}_J \cup d\varphi_{t_0}(\bar \Omega^{x}_I) , v , r) \\
& = \lim_{r \rightarrow 0^+} \theta(\bar \Omega^{y}_J , v , r) + \lim_{r \rightarrow 0^+} \theta(d\varphi_{t_0}(\bar \Omega^{x}_I), v, r)    \geq 2/3 + 1/2 > 1 ,
\end{align*}
a contradiction. In the first equality above we used that the blow-up $\Omega^{t_0,y}$ of $\Omega^{t_0}$ at $y$ corresponding to the sequence $\lambda_{j_k}$ exists and satisfies $\Omega^{t_0,y}_j = \Omega^y_j$ for all $j \in J$ and $\Omega^{t_0,y}_i = d\varphi_{t_0}(\Omega^x_i)$ for all $i \in I$, so all of these open cells are pairwise disjoint.  In particular, $U^y_J := \interior \bar \Omega^{y}_J$ and $U^y_I := \interior d \varphi_{t_0}(\bar \Omega^x_I)$ are disjoint, and $\Sigma^y \subset \bar \Omega^{y}_J = \overline{U^y_J}$ and $d\varphi_{t_0}(\Sigma^x) \subset d \varphi_{t_0}(\bar \Omega^x_I) = \overline{U^y_I}$ (where we used that $\interior \overline{\Omega^z_{\ell}} = \Omega^z_{\ell}$, $z \in \{x,y\}$, by Lemma \ref{lem:good-cells}). 

Since $v$ lies in the intersection $\Sigma^{y,1} \cap d\varphi_{t_0}(\Sigma^{x,1})$ of the regular parts of $\Sigma^y$ and $d\varphi_{t_0}(\Sigma^{x})$, which are both minimal hypersurfaces (by stationarity) lying on one side of another (by the preceding sentence), it follows by the usual smooth Hopf maximum principle \cite[Theorem 10.4]{MorganBook5Ed} that $\Sigma^{y,1}$ and $d\varphi_{t_0}(\Sigma^{x,1})$ must coincide in a neighborhood of $v$. In particular, 
\begin{equation} \label{eq:loss}
\H^{n-1}(\Sigma^y \cap d\varphi_{t_0}(\Sigma^x)) > 0 . 
\end{equation}

This now contradicts (\ref{eq:no-touch}) --- there are several ways to see this. 
Let $Z^{t_0,y}$ denote the boundary varifold corresponding to the blow-up $\Omega^{t_0,y}$ defined above, and let $\Sigma^{t_0,y} := \supp \norm{Z^{t_0,y}} = \Sigma^y \cup d\varphi_{t_0}(\Sigma^x)$. A somewhat indirect way is to use that (\ref{eq:no-touch}) implies that  $Z^{t_0,y} = Z^y + (d\varphi_{t_0})_*(Z^x)$ as $(n-1)$ varifolds (because $Z^{t_0,y,\lambda_{j_k}} = Z^{y,\lambda_{j_k}} + (d\varphi_{t_0})_*(Z^{x,\lambda_{j_k}})$ locally for all $k \gg 1$),  
but (\ref{eq:loss}) then implies that $Z^{t_0,y}$ is no longer multiplicity one, in contradiction to Theorem \ref{thm:GMT}. A more direct way is to use that (\ref{eq:loss}) implies the strict inequality
\[
 \Theta(\Sigma^{t_0},y) = \Theta(\Sigma^{t_0,y},0) < \Theta(\Sigma^y,0) + \Theta(d\varphi_{t_0}(\Sigma^x),0) = \Theta(\Sigma_J,y) + \Theta(\Sigma_I,x) ,
 \]
 meaning that at the time of collision, a positive $\H^{n-1}$-measure of boundary has been lost in a neighborhood of $y$, contradicting (\ref{eq:no-touch}). This means that our initial assumption that $\Sigma$ is disconnected is impossible, thereby concluding the proof. 
 \end{proof}

Corollary \ref{cor:intro-main} follows immediately from Theorem \ref{thm:intro-main} and the following lemma. While it is well-known to experts, we could not find a proof, so we provide one here for completeness. 

\begin{lemma} \label{lem:no-touch}
Let $\Omega=(\Omega_1,\Omega_2)$ be a single-bubble isoperimetric minimizing
cluster on a smooth complete Riemannian manifold
$(M^n,g)$, with the representatives fixed as in
Lemma \ref{lem:good-cells}. Then distinct connected components
of each $\Omega_i$ have disjoint closures.
\end{lemma}

\begin{proof}
Write $\Sigma=\partial\Omega_1=\partial\Omega_2$.
If $n=1$, $\Sigma$ is locally finite, and two components meeting
at a point would occupy both sides of that point, contradicting
$\Omega_i=\interior\overline{\Omega_i}$. Hence assume $n\geq2$. 
Without loss of generality, we establish the claim for $\Omega_2$ (which may have infinite volume).  For a set $E$ of locally finite perimeter, write
$A(E;U):=\H^{n-1}(\partial^*E\cap U)$.

Let $\Theta_1$ be a connected component of
$\Omega_2$, and set $\Theta_2=\Omega_2\setminus\Theta_1$.
Assume in the contrapositive that $x\in\partial\Theta_1\cap\partial\Theta_2$. Clearly $\Omega^0=(\Omega_1 , \Theta_1,\Theta_2)$
is a $3$-partition with the same topological boundary $\Sigma$. Since $\Omega$ is locally almost perimeter minimizing, then so is $\Omega^0$ (note that it may not be an isoperimetric minimizing cluster or even a cluster since we may have $V(\Theta_1) = V(\Theta_2) = \infty$). 

Now apply Theorem \ref{thm:GMT} to $\Omega$ and $\Omega^0$ along
a common subsequence $\lambda_k \searrow 0$, obtaining
\[
\Omega^x=(\Omega_1^x,\Omega_2^x),
\qquad
\Omega^{0,x}=(\Omega_1^x,\Theta_1^x,\Theta_2^x).
\]
On every fixed compact subset of $T_x M$, the rescaled boundary
varifolds agree for all sufficiently small scales.
Theorem \ref{thm:GMT}(\ref{it:varifold-blowup}) therefore gives
a common limiting boundary varifold $Z^x$, with support
$\Sigma^x=\partial\Omega_{1}^x = \partial\Omega_2^x$.
Both limiting partitions cover exactly
$T_xM\setminus\Sigma^x$; cancelling their common cell gives
the exact disjoint union
\begin{equation} \label{eq:union}
\Omega_2^x=\Theta_1^x\cup\Theta_2^x.
\end{equation}

Since $x\in\partial\Theta_j$ and $\Theta_j$ is open,
the Infiltration Lemma
\cite[Lemma 4.5]{CicaleseLeonardiMaggi-AlmostPerimeterMinimizers}
gives $\theta(\Theta_j,x,r)\geq\eps>0$ for $j=1,2$
and all sufficiently small $r>0$.
Passing to the blow-up shows that both $\Theta_j^x$ are non-empty.
The $\Theta_j^x$'s are consequently disjoint non-empty proper
open cones, so
$0\in\partial\Theta_1^x\cap\partial\Theta_2^x$.

For every bounded open ball $U\subset T_xM$, the common
boundary varifold and \eqref{eq:A-equiv} imply
\[
\begin{aligned}
A(\Theta_1^x;U)+A(\Theta_2^x;U)
=2A(\Omega^{0,x};U)-A(\Omega_{1}^x;U)
=2A(\Omega^x;U)-A(\Omega_{1}^x;U)
=A(\Omega_2^x;U).
\end{aligned}
\]
Since $\Omega^x$ is a perimeter-minimizing two-cell partition,
$\Omega_2^x$ is perimeter minimizing as a set.
Fix  $\{j_1,j_2\} = \{1,2\}$ and a set $F$ of locally finite perimeter
with $F\Delta\Theta_{j_1}^x\Subset U$. The exact union (\ref{eq:union}) gives
\[
(F\cup\Theta_{j_2}^x)\Delta\Omega_2^x
\subseteq F\Delta\Theta_{j_1}^x\Subset U.
\]
Thus perimeter subadditivity yields
\[
\begin{aligned}
A(\Theta_{j_1}^x;U)+A(\Theta_{j_2}^x;U) 
=A(\Omega_2^x;U)
\leq A(F\cup\Theta_{j_2}^x;U)
\leq A(F;U)+A(\Theta_{j_2}^x;U).
\end{aligned}
\]
Cancelling the finite last term proves perimeter minimality
of each $\Theta_j^x$, and hence $Z^x_{j} :=\sabs{\partial^*\Theta_{j}^x}$ is stationary. 
Theorem \ref{thm:GMT} also gives $\supp\snorm{Z^x_{j}}=\partial\Theta_{j}^x$.

Finally, at any $v\in\Sigma^{x,1}$, choose a neighborhood $W$
in which $\Sigma^x$ is a smooth hypersurface and
$W\cap\Omega_2^x$ is one connected side.
Its partition into the disjoint relatively open sets
$W\cap\Theta_1^x$ and $W\cap\Theta_2^x$ forces one to be empty.
Hence $v$ cannot belong to both boundaries.
Moreover, $\Sigma^{x,2}=\emptyset$; indeed, a $\Y$-point would require
three distinct labels for its three pairwise adjacent sectors,
whereas $\Omega^x$ has only two cells. Consequently,
\[
\partial\Theta_1^x\cap\partial\Theta_2^x
\subseteq\Sigma^{x,\geq3},
\qquad
\H^{n-2}(\partial\Theta_1^x\cap\partial\Theta_2^x)=0
\]
by Theorem \ref{thm:GMT}(\ref{it:CES}).
Theorem \ref{thm:SMP} applied to $Z^x_1,Z^x_2$ now implies that
their supports must be disjoint, contradicting their common
point $0$.
\end{proof}

\bibliographystyle{plain}
\bibliography{../../../ConvexBib}

\def\cprime{$'$} \def\textasciitilde{$\sim$}
\begin{thebibliography}{10}

\bibitem{Allard-FirstVariationOfVarifold}
W.~K. Allard.
\newblock On the first variation of a varifold.
\newblock {\em Ann. of Math. (2)}, 95:417--491, 1972.

\bibitem{AlmgrenMemoirs}
F.~J. Almgren, Jr.
\newblock Existence and regularity almost everywhere of solutions to elliptic
  variational problems with constraints.
\newblock {\em Mem. Amer. Math. Soc.}, 4(165), 1976.

\bibitem{BayleRosales}
V.~Bayle and C.~Rosales.
\newblock Some isoperimetric comparison theorems for convex bodies in
  {R}iemannian manifolds.
\newblock {\em Indiana Univ. Math. J.}, 54(5):1371--1394, 2005.

\bibitem{BellettiniKholmatov2018}
G.~Bellettini and S.~Yu. Kholmatov.
\newblock Minimizing movements for mean curvature flow of partitions.
\newblock {\em SIAM Journal on Mathematical Analysis}, 50(4):4117--4148, 2018.

\bibitem{BenjaminiCao}
I.~Benjamini and J.~Cao.
\newblock A new isoperimetric comparison theorem for surfaces of variable
  curvature.
\newblock {\em Duke Math. J.}, 85(2):359--396, 1996.

\bibitem{DeltaHomogeneousManifolds}
V.~N. Berestovski\u{\i} and Yu.~G. Nikonorov.
\newblock On {$\delta$}-homogeneous {R}iemannian manifolds.
\newblock {\em Differential Geom. Appl.}, 26(5):514--535, 2008.

\bibitem{HomogeneousGeodesicsBook}
V.~N. Berestovski\u{\i} and Yu.~G. Nikonorov.
\newblock {\em Riemannian manifolds and homogeneous geodesics}.
\newblock Springer Monographs in Mathematics. Springer, Cham, [2020] \copyright
  2020.

\bibitem{Castro-Products}
K.~Castro.
\newblock Isoperimetric inequalities in cylinders with density.
\newblock {\em Nonlinear Anal.}, 217:Paper No. 112726, 10, 2022.

\bibitem{CicaleseLeonardiMaggi-AlmostPerimeterMinimizers}
M.~Cicalese, G.~P. Leonardi, and F.~Maggi.
\newblock Improved convergence theorems for bubble clusters. i. the planar
  case.
\newblock {\em Indiana University Mathematics Journal}, 65(6):1979--2050, 2016.

\bibitem{CES-RegularityOfMinimalSurfacesNearCones}
M.~Colombo, N.~Edelen, and L.~Spolaor.
\newblock The singular set of minimal surfaces near polyhedral cones.
\newblock {\em J. Differential Geom.}, 120(3):411--503, 2022.

\bibitem{CottonFreeman-DoubleBubbleInSandH}
A.~Cotton and D.~Freeman.
\newblock The double bubble problem in spherical space and hyperbolic space.
\newblock {\em Int. J. Math. Math. Sci.}, 32(11):641--699, 2002.

\bibitem{ConnectednessOfBoundaryAndComplement}
A.~Czarnecki, M.~Kulczycki, and W.~Lubawski.
\newblock On the connectedness of boundary and complement for domains.
\newblock {\em Ann. Polon. Math.}, 103(2):189--191, 2012.

\bibitem{Gonzalo-Products}
J.~Gonzalo.
\newblock Soap bubbles and isoperimetric regions in the product of a closed
  manifold with {E}uclidean space.
\newblock arxiv.org/abs/1312.6311, 2013.

\bibitem{Hebey-NonlinearAnalysisBook}
E.~Hebey.
\newblock {\em Nonlinear analysis on manifolds: {S}obolev spaces and
  inequalities}, volume~5 of {\em Courant Lecture Notes in Mathematics}.
\newblock New York University, Courant Institute of Mathematical Sciences, New
  York; American Mathematical Society, Providence, RI, 1999.

\bibitem{Hutchings-StructureOfDoubleBubbles}
M.~Hutchings.
\newblock The structure of area-minimizing double bubbles.
\newblock {\em J. Geom. Anal.}, 7(2):285--304, 1997.

\bibitem{Ilmanen-SMP}
T.~Ilmanen.
\newblock A strong maximum principle for singular minimal hypersurfaces.
\newblock {\em Calc. Var. Partial Differential Equations}, 4(5):443--467, 1996.

\bibitem{MaggiBook}
F.~Maggi.
\newblock {\em Sets of finite perimeter and geometric variational problems: an
  introduction to {G}eometric {M}easure {T}heory}, volume 135 of {\em Cambridge
  Studies in Advanced Mathematics}.
\newblock Cambridge University Press, Cambridge, 2012.

\bibitem{EMilman-Slabs}
E.~Milman.
\newblock Isoperimetric inequalities on slabs with applications to cubes and
  {G}aussian slabs.
\newblock {\em to appear in Comm. Pure Appl. Math.}
\newblock arXiv:2403.06602.

\bibitem{EMilmanNeeman-GaussianMultiBubble}
E.~Milman and J.~Neeman.
\newblock The {G}aussian double-bubble and multi-bubble conjectures.
\newblock {\em Ann. of Math. (2)}, 195(1):89--206, 2022.

\bibitem{EMilmanNeeman-TripleAndQuadruple}
E.~Milman and J.~Neeman.
\newblock The structure of isoperimetric bubbles on $\mathbb{R}^n$ and
  $\mathbb{S}^n$.
\newblock {\em Acta Math.}, 234(1):71--188, 2025.

\bibitem{MorganSoapBubblesInR2}
F.~Morgan.
\newblock Soap bubbles in {${\bf R}^2$} and in surfaces.
\newblock {\em Pacific J. Math.}, 165(2):347--361, 1994.

\bibitem{MorganBook5Ed}
F.~Morgan.
\newblock {\em Geometric measure theory (a beginner's guide)}.
\newblock Elsevier/Academic Press, Amsterdam, fifth edition, 2016.

\bibitem{FloresNardulli-ContinuityOfProfile}
A.~E. Mu\~{n}oz Flores and S.~Nardulli.
\newblock Local {H}\"{o}lder continuity of the isoperimetric profile in
  complete noncompact {R}iemannian manifolds with bounded geometry.
\newblock {\em Geom. Dedicata}, 201:1--12, 2019.

\bibitem{Munkres-TopologyBook2ndEd}
J.~Munkres.
\newblock {\em Topology}.
\newblock Prentice Hall, second edition, 2000.

\bibitem{NaberValtorta-MinimizingHarmonicMaps}
A.~Naber and D.~Valtorta.
\newblock Rectifiable-{R}eifenberg and the regularity of stationary and
  minimizing harmonic maps.
\newblock {\em Ann. of Math. (2)}, 185(1):131--227, 2017.

\bibitem{NaberValtorta-SingularSet}
A.~Naber and D.~Valtorta.
\newblock The singular structure and regularity of stationary varifolds.
\newblock {\em J. Eur. Math. Soc. (JEMS)}, 22(10):3305--3382, 2020.

\bibitem{Pedrosa-SphericalCylinders}
R.~H.~L. Pedrosa.
\newblock The isoperimetric problem in spherical cylinders.
\newblock {\em Ann. Global Anal. Geom.}, 26(4):333--354, 2004.

\bibitem{PetersenBook2ndEd}
P.~Petersen.
\newblock {\em Riemannian geometry}, volume 171 of {\em Graduate Texts in
  Mathematics}.
\newblock Springer, New York, second edition, 2006.

\bibitem{Pittet-ProfileOfHomogeneousManifolds}
Ch. Pittet.
\newblock The isoperimetric profile of homogeneous {R}iemannian manifolds.
\newblock {\em J. Differential Geom.}, 54(2):255--302, 2000.

\bibitem{Resende-ClustersInBoundedGeometry}
R.~Resende~de Oliveira.
\newblock On clusters and the multi-isoperimetric profile in {R}iemannian
  manifolds with bounded geometry.
\newblock {\em J. Dyn. Control Syst.}, 29(2):419--441, 2023.

\bibitem{Ritore-IsoperimetricBook}
M.~Ritor\'{e}.
\newblock {\em Isoperimetric inequalities in {R}iemannian manifolds}, volume
  348 of {\em Progress in Mathematics}.
\newblock Birkh\"{a}user/Springer, Cham, [2023] \copyright 2023.

\bibitem{RitoreVernadakis-Products}
M.~Ritor\'{e} and E.~Vernadakis.
\newblock Large isoperimetric regions in the product of a compact manifold with
  {E}uclidean space.
\newblock {\em Adv. Math.}, 306:958--972, 2017.

\bibitem{Simon-Sigma2}
L.~Simon.
\newblock Cylindrical tangent cones and the singular set of minimal
  submanifolds.
\newblock {\em J. Differential Geom.}, 38(3):585--652, 1993.

\bibitem{Simon-IntroToGMT}
L.~Simon.
\newblock Introduction to {G}eometric {M}easure {T}heory.
\newblock Revised and updated version of ``{L}ectures on geometric measure
  theory", ANU, Canberra, 1983, available at
  \href{https://math.stanford.edu/~lms/ntu-gmt-text.pdf}{https://math.stanford.edu/~lms/ntu-gmt-text.pdf},
  2018.

\bibitem{SternbergZumbrun}
P.~Sternberg and K.~Zumbrun.
\newblock On the connectivity of boundaries of sets minimizing perimeter
  subject to a volume constraint.
\newblock {\em Comm. Anal. Geom.}, 7(1):199--220, 1999.

\bibitem{Wickra-SMP}
N.~Wickramasekera.
\newblock A general regularity theory for stable codimension 1 integral
  varifolds.
\newblock {\em Ann. of Math. (2)}, 179(3):843--1007, 2014.

\end{thebibliography}

\end{document}